\documentclass[reqno]{amsart}

\usepackage{graphicx,subfigure,color}

\numberwithin{equation}{section}

\usepackage[latin1]{inputenc}
\usepackage[english]{babel}

\usepackage{amsmath,amsthm,amsfonts,latexsym,amssymb}
\usepackage[colorlinks]{hyperref}
\hypersetup{linkcolor=blue,citecolor=blue,filecolor=black,urlcolor=blue}
\usepackage{comment}

\usepackage{color}
{ \theoremstyle{plain}
\newtheorem{theorem}{Theorem}[section]

\newtheorem{lemma}[theorem]{Lemma}
\newtheorem{corollary}[theorem]{Corollary}
  \theoremstyle{remark}
\newtheorem{remark}[theorem]{Remark}
  \theoremstyle{definition}

}

\def\R{\mathbb{R}}
\def\Z{\mathbb{Z}}

\begin{document}
\subjclass[2010]{35J92, 35A24, 35B05, 35B09.}

\keywords{Quasilinear elliptic equations, Shooting method, Sobolev-supercritical nonlinearities, Neumann boundary conditions.}

\title[]{
Multiple positive solutions for a class of $p$-Laplacian Neumann problems without growth conditions}

\author[A. Boscaggin]{Alberto Boscaggin}
\address{Alberto Boscaggin\newline\indent 
Dipartimento di Matematica
\newline\indent
Università di Torino
\newline\indent
via Carlo Alberto 10, 10123 Torinio, Italia}
\email{alberto.boscaggin@unito.it}

\author[F. Colasuonno]{Francesca Colasuonno}
\address{Francesca Colasuonno\newline\indent
D\'epartement de Math\'ematique\newline\indent
Universit\'e Libre de Bruxelles
\newline\indent
Campus de la Plaine - CP214\newline\indent
boulevard du Triomphe - 1050 Bruxelles, Belgique
}
\email{francesca.colasuonno@unibo.it}

\author[B. Noris]{Benedetta Noris}
\address{Benedetta Noris
\newline \indent Laboratoire Ami\'enois de Math\'ematique Fondamentale et Appliqu\'ee\newline\indent
Universit\'e de Picardie Jules Verne\newline\indent
33 rue Saint- Leu, 80039 AMIENS, France}
\email{benedetta.noris@u-picardie.fr}

\date{\today}

\begin{abstract} 
For $1<p<\infty$, we consider the following problem
$$
-\Delta_p u=f(u),\quad u>0\text{ in }\Omega,\quad\partial_\nu u=0\text{ on }\partial\Omega,
$$
where $\Omega\subset\mathbb R^N$ is either a ball or an annulus. The nonlinearity $f$ is possibly supercritical in the sense of Sobolev embeddings; in particular our assumptions allow to include the prototype nonlinearity $f(s)=-s^{p-1}+s^{q-1}$ for every $q>p$.
We use the shooting method to get existence and multiplicity of non-constant radial solutions.  
With the same technique, we also detect the oscillatory behavior of the solutions around the constant solution $u\equiv1$. 
In particular, we prove a conjecture proposed in [D. Bonheure, B. Noris, T. Weth, {\it Ann. Inst. H. Poincar\'e Anal. Non Lin\'aire} vol. 29, pp. 573-588 (2012)], that is to say, if $p=2$ and $f'(1)>\lambda_{k+1}^{\textnormal{rad}}$, there exists a radial solution of the problem having exactly $k$ intersections with $u\equiv1$ for a large class of nonlinearities.
\end{abstract}

\maketitle

\section{Introduction}
\subsection{Assumptions and main results}
The aim of this paper is to investigate the existence of solutions of the following $p$-Laplacian Neumann problem
\begin{equation}\label{main}
\left\{
\begin{array}{ll}
\vspace{0.1cm}
-\Delta_p u = f(u) & \mbox{ in } \Omega \\
\vspace{0.1cm}
u > 0 & \mbox{ in } \Omega \\
\partial_\nu u = 0 & \mbox{ on } \partial\Omega, \\
\end{array}
\right.
\end{equation}
where $1<p<\infty$, $\nu$ is the outer unit normal of $\partial \Omega$, and we require very mild assumptions on $f$, which allow in particular to consider 
\begin{equation}\label{eq:prototype}
f(s)=-s^{p-1}+s^{q-1} \quad\text{ for every }q>p
\end{equation}
as a prototype nonlinearity.
Our main purpose is not to impose any growth conditions on $f(s)$ as $s\to\infty$, so that $f$ may have a supercritical behavior with respect to the critical Sobolev exponent (that is to say, $N>p$ and $q>Np/(N-p)$ in the prototype nonlinearity \eqref{eq:prototype}).

We work in a radial domain $\Omega\subset\mathbb R^N$, $N\geq1$, which is either an annulus 
$$
\Omega=\mathcal A(R_1,R_2):=\{ x \in \mathbb{R}^N \, : \, R_1 < \vert x \vert < R_2 \},\quad 0<R_1<R_2<\infty,
$$
or a ball
$$
\Omega=\mathcal B(R_2):=\{ x \in \mathbb{R}^N \, : \vert x \vert < R_2 \},\quad 0=R_1<R_2<\infty,
$$
and we look for radial solutions of \eqref{main}. Throughout the paper, with abuse of notation, we denote $u(r):=u(x)$ for all $|x|=r$

We assume that $f$ satisfies the following conditions:
\begin{itemize}
\item[$(f_\textrm{reg})$] $f \in \mathcal{C}([0,\infty)) \cap \mathcal{C}^1((0,\infty))$;
\item[$(f_\textrm{eq})$] $f(0) = f(1) = 0$, $f(s) < 0$ for $0 < s < 1$ and $f(s) > 0$ for $s > 1$; 
\item[$(f_0)$] there exists $C_0 \in [0,\infty)$ such that $\lim_{s\to0^+} \frac{f(s)}{s^{p-1}}  =- C_0$;
\item[$(f_1)$] there exists $C_1 \in [0,\infty]$ such that $\lim_{s\to1} \frac{f(s)}{\vert s - 1 \vert^{p-2} (s-1)} = C_1$.
\end{itemize}
We remark that the choice of the constant $1$ in $(f_\textrm{eq})$ is arbitrary, this constant could be replaced by any $s_1\in (0,\infty)$, thus changing accordingly $(f_1)$. We also stress that we do not impose at infinity any of the conditions frequently used in the literature, such as the Ambrosetti-Rabinowitz one.

\begin{remark}\label{remhp}
Let us notice that the assumptions $(f_\textrm{reg})$ and $(f_1)$ are not independent; indeed, the differentiability of $f$  at $s=1$ implies $C_1 \in [0,\infty)$ if $p=2$ and $C_1 = 0$ if $1<p<2$. We believe that this regularity condition can be removed by an approximation argument; however, since it is satisfied for the model nonlinearity \eqref{eq:prototype}, we have preferred to avoid this technical step. We also observe that, in view of $(f_\textrm{eq})$, the ratio $f(s)/s^{p-1}$ appearing in hypothesis $(f_0)$ is negative for $s \to 0^+$, so that $\limsup_{s \to 0^+}f(s)/s^{p-1}\leq 0$; actually, 
a careful inspection of the proofs shows that, in all the results below, $(f_0)$ could be replaced by the weaker assumption
$\liminf_{s \to 0^+} f(s) / s^{p-1} > -\infty$, that is to say, $f(s)/s^{p-1}$ is bounded in a right neighborhood of $s = 0$.
\end{remark}

In order to state our main results, let us introduce $\lambda_{k}^{\textnormal{rad}}$ as the $k$-th radial eigenvalue of $-\Delta_p u = \lambda|u|^{p-2}u$ in $\Omega$ with Neumann boundary conditions, i.e. $0=\lambda_1^{\textnormal{rad}}<\lambda_2^{\textnormal{rad}}<\lambda_3^{\textnormal{rad}}<\dots$, cf. Section \ref{A} for further details. In case the constant $C_1$ appearing in assumption $(f_1)$ is positive, we have the following existence and multiplicity result.

\begin{theorem}\label{th:main}
Let $\Omega$ be either the annulus $\mathcal A(R_1,R_2)$ or the ball $\mathcal B(R_2)$ and let $f$ satisfy $(f_{\mathrm{reg}})$-$(f_1)$. 

Assume that $C_1 > \lambda_{k+1}^{\textnormal{rad}}$ for some integer $k \geq 1$. Then, there exist at least $k$ non-constant radial solutions $u_1,\ldots,u_k$ to \eqref{main}. Moreover, $u_j(r)-1$ has exactly $j$ zeros for $r\in(R_1,R_2)$, for every $j=1,\ldots,k$.

In particular, if $C_1 = +\infty$, then \eqref{main} has infinitely many non-constant radial solutions.
\end{theorem}

Noting that in the case $p=2$ we have $C_1=f'(1)$, Theorem \ref{th:main}
shows that the conjecture proposed in \cite{BNW} holds true, that is to say, if $f'(1)>\lambda_{k+1}^{\textnormal{rad}}$ for some integer $k\geq1$, there exists a radial solution of \eqref{main} having exactly $k$ intersections with the constant solution $u\equiv1$. More precisely, taking into account also Remark \ref{remhp}, we can state the following general result.

\begin{corollary}\label{conjecturesolved}
Let $\Omega$ be either the annulus $\mathcal A(R_1,R_2)$ or the ball $\mathcal B(R_2)$ and let $f$ satisfy $(f_{\mathrm{reg}})$-$(f_\mathrm{eq})$. If $f(s)/s$ is bounded in a right neighborhood of $s=0$ and $f'(1) >\lambda_{k+1}^{\textnormal{rad}}$, then there exist a radial solution $u$ of 
$$
\left\{
\begin{array}{ll}
\vspace{0.1cm}
-\Delta u = f(u) & \mbox{ in } \Omega \\
\vspace{0.1cm}
u > 0 & \mbox{ in } \Omega \\
\partial_\nu u = 0 & \mbox{ on } \partial\Omega, \\
\end{array}
\right.
$$
such that $u(r)-1$ has exacly $k$ zeros for $r \in (R_1,R_2)$.
\end{corollary}
 
When $C_1=0$, a different behavior appears. First, the existence of non-constant solutions of \eqref{main} depends on the diameter of the domain. Secondly, if the diameter is sufficiently large, there exist now \textit{two} solutions having the same oscillatory behavior, in the sense specified in the following theorem.

\begin{theorem}\label{th:main2}
Let $f$ satisfy $(f_{\mathrm{reg}})$-$(f_1)$ with $C_1 = 0$.
\begin{itemize}
\item[(i)] For any integer $k \geq 1$ there exists $R_*(k) > 0$ such that if $R_2 > R_*(k)$, then problem \eqref{main} in $\Omega=\mathcal B(R_2)$ has at least $2k$ non-constant radial solutions. 
\item[(ii)] For any integer $k \geq 1$ and any $\varepsilon > 0$ there exists $R_*(k,\varepsilon) > 0$ such that if $R_1 < \varepsilon R_2$ and $R_2 > R_*(k,\varepsilon)$, then problem \eqref{main} in $\Omega=\mathcal A(R_1,R_2)$ has at least $2k$ non-constant radial solutions. 
\end{itemize}
Denoting these solutions by $u_1^+,\ldots,u_{k}^+$, $u_1^-,\ldots,u_{k}^-$, we have that each $u_j^\pm(r)-1$ has exactly $j$ zeros for $r\in(R_1,R_2)$, for every $j=1,\ldots,k$.
\end{theorem} 

Noting that the prototype nonlinearity \eqref{eq:prototype} satisfies the assumptions $(f_{\mathrm{reg}})$-$(f_1)$ with
$$
C_0=1, \qquad 
C_1=\left\{\begin{array}{ll}
0 & \text{ if } 1<p<2, \\
q-2 & \text{ if } p=2, \\
+\infty & \text{ if } p>2, \\
\end{array}\right
.
$$
we have the following corollary of Theorems \ref{th:main} and \ref{th:main2}.
 
\begin{corollary}\label{cor:modello}
Let $\Omega$ be either the annulus $\mathcal A(R_1,R_2)$ or the ball $\mathcal B(R_2)$, and consider the Neumann problem
\begin{equation}\label{modello}
\left\{
\begin{array}{ll}
\vspace{0.1cm}
-\Delta_p u + u^{p-1} = u^{q-1}  & \mbox{ in } \Omega, \\
\vspace{0.1cm}
u > 0 & \mbox{ in } \Omega, \\
\partial_\nu u = 0 & \mbox{ on } \partial\Omega, \\
\end{array}
\right.
\end{equation}
with $q > p$. Then:
\begin{itemize}
\item[(i)] for $p > 2$, \eqref{modello} has infinitely many non-constant radial solutions;
\item[(ii)] for $p = 2$ and $q-2 > \lambda_{k+1}^{\textnormal{rad}}$ for some $k\ge1$, \eqref{modello} has at least $k$ non-constant radial solutions;
\item[(iii)] for $1 < p < 2$, for any integer $k \geq 1$ and any $\varepsilon > 0$ there exists $R_*(k,\varepsilon) > 0$ such that if $R_1 < \varepsilon R_2$ and $R_2 > R_*(k,\varepsilon)$, then problem \eqref{modello} in $\Omega=\mathcal A(R_1,R_2)$ has at least $2k$ non-constant radial solutions. Analogously in the case $\Omega=\mathcal B(R_2)$.
\end{itemize}
\end{corollary}

We remark that all the solutions found in this paper satisfy $u(0)<1$ and are increasing near the origin, see Section \ref{sec:simulations} for some open problems concerning solutions with $u(0)>1$. For the special case in which the nonlinearity is a small perturbation of the exponential function, solutions with $u(0)>1$ are found in \cite{PistoiaVaira2015}.

\subsection{Pre-existing literature}
Semilinear and quasilinear elliptic equations with Sobolev-critical and supercritical growth have been extensively studied in the literature, but mainly coupled with Dirichlet boundary conditions. H. Brezis adresses to Neumann problems Section 6.4 of his survey on Sobolev-critical equations \cite{brezis2014nonlinear}, saying that little is known in this case. 
The first result of which we are aware concerning Neumann boundary conditions is the one by Lin and Ni in \cite{LinNi88}. The authors consider the equation \eqref{modello} with $p=2$, $q>2N/(N-2)$ and $\Omega=\mathcal B(R_2)$, and prove that for $R_2$ sufficiently small \eqref{modello} only admits the constant solution, whereas for $R_2$ sufficiently large there exists a non-constant solution. We also refer to \cite{LinNiTakagi88} for similar results in the case $q<2N/(N-2)$. When $q$ is critical, this kind of result is no longer true. Indeed, in \cite{AdimurthiYadava}, Adimurthi and Yadava prove that in dimensions $N=4,\,5,\,6$ there exists a decreasing solution in balls of small radius. This depends not only on the dimension $N$, see \cite{AdimurthiYadava97,budd1991asymptotic}, but also on the shape of the domain, see \cite{WangWeiYan2010}. We also wish to mention that \cite{AdimurthiYadava97} is the first paper where \eqref{modello} is studied for $p\neq2$.

As soon as $N>p$, $q>Np/(N-p)$ and $\Omega=\mathcal B(R_2)$, the absence of Sobolev embeddings prevents from treating \eqref{modello} with the standard variational techniques.
Of course, the choice of working in a radial setting allows to gain some compactness, but not enough, for example, to define the Euler-Lagrange functional associated to the equation.
Recently, some techniques have been proposed to overcome this lack of compactness.
Different methods have been introduced simultaneously and independently, for $p=2$, in \cite{BonheureSerra2011}, \cite{GrossiNoris} and  \cite{SerraTilli2011}. In particular, in \cite{SerraTilli2011}, Serra and Tilli get over the lack of compactness by considering the cone of non-negative, non-decreasing radial functions of $H^1(\Omega)$. This technique proved to be quite powerful and has been adopted in many of the subsequent papers that we are going to illustrate. Serra and Tilli prove that, if $g$ satisfies some suitable assumptions and $a(|x|)>0$ is a non-decreasing and non-constant weight, then the radial problem 
\[
-\Delta u+u=a(|x|)g(u), \quad u>0 \text{ in } \mathcal B(R), \quad \partial_\nu u=0 \text{ on } \partial \mathcal B(R)
\]
admits at least one radially increasing solution. Secchi generalises this result to the case $p\neq 2$ in \cite{secchi2012increasing} using the same assumptions on $g$ e $a$. 

In \cite{BNW} and in \cite{ColasuonnoNoris} the authors consider the case $a$ constant, respectively in the cases $p=2$ and $p>2$ (see also \cite{ma2016radial}). The additional difficulty is now to prove that the solution found is itself non-constant, and this can be done under an extra condition on $g$, namely $(g_3)$ below. 

\begin{theorem}[{\cite[Theorem 1.3]{BNW},\cite[Theorem 1.1]{ColasuonnoNoris}}]\label{thm:intro}
Let $p\geq2$ and let $g:[0,\infty)\to\mathbb R$ be of class $C^1([0,\infty))$ and satisfy
\begin{itemize}
\item[$(g_1)$] $\lim_{s\to 0^+}\frac{g(s)}{s^{p-1}}\in[0,1)$;
\item[$(g_2)$] $\liminf_{s\to\infty}\frac{g(s)}{s^{p-1}}>1$;
\item[$(g_3)$] there exists a constant $u_0>0$ such that $g(u_0)=u_0^{p-1}$ and
$g'(u_0)>\lambda_{2}^{\textnormal{rad}}+1$ if $p=2$, or $g'(u_0)>(p-1)u_0^{p-2}$ if $p>2$.
\end{itemize}
Then there exists a non-constant, radial, non-decreasing solution of
\begin{equation}\label{eq:intro}
-\Delta_p u+u^{p-1}=g(u),\quad u>0\text{ in }\mathcal B(R),\quad 
\partial_\nu u=0\text{ on }\partial \mathcal B(R).
\end{equation}
\end{theorem}

Let us first comment the case $p=2$. We notice that, in the semilinear case, condition $(g_3)$ involves the second radial eigenvalue of $-\Delta$ with Neumann boundary conditions. In fact, the authors in \cite{bonheure2016multiple} show that a bifurcation phenomenon is underlying the existence result, at least in the case of the prototype nonlinearity $g(u)=|u|^{q-2}u$. They prove that at $q-2=\lambda_{k+1}^{\textnormal{rad}}$, $k\geq1$, a new branch of solutions bifurcates from the constant branch $u\equiv1$. This nontrivial branch consists of solutions having exactly $k$ oscillations around the constant solution $1$. We also refer to \cite{BonheureGrossiNorisTerracini2015}, \cite{BonheureCasterasNoris2016}, \cite{BonheureCasterasNoris2017} for other results about this class of problems. As mentioned above, it was conjectured in \cite{BNW} that a similar behavior should hold also for a general nonlinearity. If we consider a nonlinearity $f$ related to $g$ by $f(s)=g(s)-s$, the conjecture asserts that, if $f'(u_0)>\lambda_{k+1}^{\textnormal{rad}}$, $k\geq1$, there should exist a radial solution of \eqref{main} having exactly $k$ intersections with the constant solution $u_0$. 
For $f$ asymptotically linear (and hence Sobolev subcritical), this conjecture was proved in \cite{ma2016bonheure}. By means of bifurcation techniques, the authors show that, if $f'(u_0)>\lambda_{k+1}^{\textnormal{rad}}$ for some $k\ge 1$, then there exist at least $2k$ different non-constant solutions of \eqref{eq:intro}, $k$ of them are increasing and $k$ decreasing in a neighborhood of zero. 
In the present paper (cf. Corollary \ref{conjecturesolved}) we are able to provide a complete proof of the conjecture, without assuming any growth conditions at infinity on $f$. 

We remark that the assumptions $(f_{\mathrm{reg}})$-$(f_1)$ are substantially more general than $(g_1)$-$(g_3)$. Indeed, we have $f(s)=-s^{p-1}+g(s)$ and $(g_1)$ requires that the constant $C_0$ defined in $(f_0)$ belongs to $(0,1]$, $(g_3)$ requires that, when $p > 2$, $C_1$ defined in $(f_1)$ satisfies $C_1=+\infty$, and $(g_2)$ is equivalent to $\liminf_{s\to+\infty}f(s)>0$. In addition, in the present paper we find infinitely many solutions of \eqref{main} in the case $C_1=+\infty$, whereas in \cite{ColasuonnoNoris} only one solution was found (the non-decreasing one, which we can now prove being indeed strictly increasing, see \eqref{polari}-\eqref{sys2} below). Indeed, to the best of our knowledge, no multiplicity results were known for problem \eqref{main} in the case $p\neq2$. In particular, for $p<2$ we obtain here a multiplicity result which is completely new in the literature. As already noticed, the behavior for $C_1=0$ (corresponding to $p<2$ for the prototype nonlinearity \eqref{eq:prototype}) is different from the one for $C_1>0$, since the existence of solutions depends on the diameter of the domain, and solutions always come in couples, so that we find two solutions with the same oscillatory behavior. In this regard, see also the numerical simulations in Section \ref{sec:simulations}.

For results in a non-radial setting (in the case $p=2$), we refer to the recent works \cite{CowanMoameni,delPinoPistoiaVaira2016}.
We also wish to mention the generalisations to systems considered in \cite{bonheure2013radial,MaChenWang} and the extensive literature concerning concentrating solutions for supercritical Neumann problem with a perturbation parameter, see for example \cite{MalchiodiMontenegro2002,MalchiodiMontenegro2004,Malchiodi2004,MalchiodiNiWei2005,
Malchiodi2005,delPino2015interior}.

\subsection{Main ideas of the proof and organization of the paper}
We adopt a shooting method: it seems indeed that such a technique turns out to be particularly effective 
when trying to identify the different multiplicity scenarios appearing on varying of $p$; moreover, it allows to
avoid most of the technical assumptions on the nonlinearity. For an application of the shooting method in a similar situaton we refer to \cite{BarutelloSecchiSerra}, where the authors consider the supercritical H{\'e}non equation with Neumann boundary conditions.

In Section \ref{preliminary} we rewrite the radial equation in \eqref{main} as the planar ODE system
\[
r^{N-1}|u'|^{p-2}u'=v, \qquad v'=-r^{N-1}f(u),
\]
(cf. \eqref{sys1}) and prove local uniqueness, continuous dependence and global continuability of solutions. The shooting method consists in studying the initial value problem $u(R_1)=1-d$, $v(R_1)=0$ and looking for values $d\in(0,1)$ such that the corresponding solution $(u_d,v_d)$ satisfies $v_d(R_2)=0$. Thanks to the local uniqueness, we can pass to polar-like coordinates $(\rho(r),\theta(r))$ around the point $(1,0)$ (see \eqref{polari}), so that the problem reduces to
\begin{equation*}
\text{find } d\in(0,1) \text{ such that } \theta_d(R_2)=k\pi_p \text{ for some } k\in\Z,
\end{equation*}
with $\pi_p$ defined in Lemma \ref{propCpSp}.

In Section \ref{A} we recall some known results concerning the associated eigenvalue problem, while Section \ref{sec:theo1} is devoted to the proof of Theorem \ref{th:main}. The main point is to show that, if $C_1>\lambda_{k+1}^{\mathrm{rad}}$, $k\geq1$, then $\theta_d(R_2)> (k+1)\pi_p$ for $d$ sufficiently close to $0$. This can be done by comparing \eqref{main} with its associated eigenvalue problem. In Section \ref{sec:theo2} we prove Theorem \ref{th:main2}; here, the key step is to show that, if $R_1,R_2$ are chosen as in the corresponding statement, then $\theta_{d}(R_2) > (k+1)\pi_p$
for some $d \in (0,1)$. This is proved by adapting to our context a phase-plane argument introduced in \cite{BosZan13}.
 Finally, in Section \ref{sec:simulations} we present some numerical simulations obtained with the software AUTO07p \cite{AUTO}, and propose some open problems.

\section{Multiplicity of solutions via the shooting method}

\subsection{Preliminary results}\label{preliminary}

First of all, as usual when dealing with positive solutions of a boundary value problem, we introduce a continuous extension $\hat f: \mathbb{R} \to \mathbb{R}$ of $f$ by setting
$$
\hat f(s):=\begin{cases}
	f(s) \quad&\text{if } s \geq 0,\\
	0 &\text{if } s < 0.
\end{cases}
$$
By taking into account the radial symmetry of \eqref{main}, we consider the following 1-dimensional problem
\begin{equation}\label{ode2}
\begin{cases}
-\left(r^{N-1}\varphi_p(u')\right)'=r^{N-1}\hat f(u)\quad\mbox{in }(R_1,R_2),\\
u'(R_1)=u'(R_2)=0,
\end{cases}
\end{equation}
where 
$$
\varphi_p(s) := \vert s \vert^{p-2} s
$$
and the prime symbol $'$ denotes the derivative with respect to $r$. 
We remark that, in the case $\Omega=\mathcal B(R_2)$, namely $R_1=0$, the boundary condition $u'(0)=0$ comes from the symmetry and the $C^{1,\alpha}(\bar\Omega)$-regularity of the solutions $u$ of \eqref{main}, cf. \cite[Theorem~2]{Lieberman} and also \cite{RW99}. For future reference, we also note that $\varphi_{p}^{-1}=\varphi_{p'}$, where $$\frac1p+\frac1{p'}=1.$$  
We first give the following maximum principle-type result (compare also with \cite[Section 2]{ManNjoZan95}). 
%
%

\begin{lemma}\label{max}
Let $u$ be a solution of \eqref{ode2}. Either $u\equiv -C$, with $C\geq0$, or $u(r) > 0$ for every $r \in [R_1,R_2]$.
\end{lemma}

\begin{proof}
	Let us first prove that either $u$ is a negative constant, or $u$ is non-negative. To this end, suppose by contradiction that $u$ is non-constant and that $u(r_0) < 0$ for some $r_0 \in (R_1,R_2)$. Let $[r^-,r^+] \subset [R_1,R_2]$ be the maximal interval containing $r_0$ such that $u(r) < 0$ for every $r \in (r^-,r^+)$. By the definition of $\hat f$, we have
\begin{equation}\label{eq:maximum_princ}
\hat f(u(r))=0 \quad \text{for every } r\in [r^-,r^+].
\end{equation}
Since $u$ is non-constant, by the equation in \eqref{ode2} we get the existence of $r_1\in (R_1,R_2)$ such that $u(r_1)>0$, that is to say, $r^-\neq R_1$ or $r^+\neq R_2$. Suppose, to fix the ideas, that $r^+\neq R_2$, so that
\begin{equation}\label{eq:maximum_princ2}
u(r^+)=0.
\end{equation}
Now we distinguish two cases: either $r^-\neq R_1$ or $r^-=R_1$. If the first case occurs, then $u(r^-)=0$. Hence, by \eqref{eq:maximum_princ2} we have $u'(r^-) \leq 0 \leq u'(r^+)$, so that, since $\varphi_p$ is non-decreasing, $\varphi_p(u'(r^-)) \leq 0 \leq \varphi_p(u'(r^+))$ and also 
\begin{equation}\label{rN-1phi}
(r^-)^{N-1}\varphi_p(u'(r^-)) \leq 0 \leq (r^+)^{N-1}\varphi_p(u'(r^+)).
\end{equation} 
Then, using the equation in \eqref{ode2}, \eqref{eq:maximum_princ} and \eqref{rN-1phi}, we obtain that $r^{N-1}\varphi_p(u'(r)) = 0$ for every $r \in [r^-,r^+]$, implying $u = 0$ in $[r^-,r^+]$ as well, a contradiction. 

If the second case occurs, that is $r^-=R_1$, then $u'(r^-)=0$. Again \eqref{eq:maximum_princ} implies $r^{N-1}\varphi_p(u'(r)) = 0$ for every $r \in [R_1,r^+]$ and hence, by \eqref{eq:maximum_princ2}, $u = 0$ in $[R_1,r^+]$, a contradiction.
	
	It remains to show that any non-negative solution of \eqref{ode2} is positive. To this aim, we observe that if the (non-negative) function $u$ vanishes but is not identically zero, then it necessarily has a double zero, that is, $u(r_0) = u'(r_0) = 0$ for some $r_0 \in [R_1,R_2]$.
By assumption $(f_0)$ and \cite[Theorem 4]{ReiWal97}-$(\delta)$, the solution of this Cauchy problem is unique and so it has to be $u\equiv 0$ on $[R_1,R_2]$, which is a contradiction. 	
\end{proof}

In view of the above lemma, from now on we will study problem \eqref{main} simply by looking for non-constant solutions of \eqref{ode2}. This will be done by using a shooting approach: we write the equation in \eqref{ode2} as the planar ODE system in $(R_1,R_2)$
\begin{equation}\label{sys1}
u' = \varphi_p^{-1}\left(\frac{v}{r^{N-1}}\right), \qquad v' = - r^{N-1}\hat f(u), 
\end{equation}
we consider the associated Cauchy problem with initial conditions
\begin{equation}\label{ci}
u(R_1) = 1-d, \qquad v(R_1) = 0,
\end{equation}
where $d \in [0,1]$, and we look for values $d \in (0,1)$ such that the corresponding solution $(u_d,v_d)$ is defined on the whole $[R_1,R_2]$ and satisfies $v_d(R_2) = 0$ (and hence $u_d'(R_2) = 0$). 

We stress that, when $\Omega = \mathcal{A}(R_1,R_2)$ (that is, when $R_1 > 0$), the initial condition $v(R_1) = 0$ plainly corresponds to $u'(R_1) = 0$;
on the other hand, when $\Omega = \mathcal B(R_2)$ (that is, when $R_1 = 0$), the situation is more delicate. Indeed, the ODE system \eqref{sys1} exhibits, for $r = R_1 = 0$, a singularity of order $r^{-\frac{N-1}{p-1}}$. 
When $p > N$, such a singularity is in $L^1$ and system \eqref{sys1} can be treated within the Carath\'eodory theory of ODEs (see, for instance, \cite{Hale}); on the contrary, for $1 < p \leq N$ this is no longer true. Nonetheless, it can be shown via fixed point arguments in Banach spaces that the Cauchy problem \eqref{sys1}-\eqref{ci} (requiring $v(R_1) = 0$) still has a (local) solution. All this is nowadays well-known (see \cite{FraLanSer96,ReiWal97}) and any solution $(u(r),v(r))$ of \eqref{sys1}-\eqref{ci} is such that $u'(R_1) = 0$ and $u$ solves the equation in \eqref{ode2} in the usual sense (namely, $u(r)$ and $r^{N-1}\varphi_p(u'(r)) = v(r)$ belong to $C^1([R_1,R_2])$ and the equation is satisfied pointwise). See also Remark \ref{cambiovariabile}.

To make our shooting procedure effective, we now prove the following result of uniqueness, continuous dependence, and global continuability.

\begin{lemma}\label{uni} 
	For any $d \in [0,1]$, the solution $(u_d,v_d)$ of \eqref{sys1}-\eqref{ci} is unique
	and can be defined on the whole $[R_1,R_2]$; moreover, 
	if $(d_n)\subset(0,1)$ is such that $d_n \to d \in [0,1]$, then $(u_{d_n}(r),v_{d_n}(r)) \to (u_d(r),v_d(r))$ uniformly in $r \in [R_1,R_2]$.
\end{lemma}

\begin{proof}
We first focus on the uniqueness; notice that by this we mean that $(u_d,v_d)$ remains unique as long as defined and, in turn, this requires us to investigate the local uniqueness of any Cauchy problem 
	$$
	u(\bar r) = \bar u, \qquad v(\bar r) = \bar v, 
	$$
associated with \eqref{sys1}, where $\bar r \in [R_1,R_2]$ and $(\bar u,\bar v) \in [0,1] \times \{0\}$ if $\bar r = 0$ and $(\bar u,\bar v) \in \mathbb{R}^2$ if $\bar r > 0.$ 
If $\bar r\neq0$, $\bar v\neq0$ and $\bar u\neq0$, this follows from the Cauchy-Lipschitz Theorem. 
Otherwise, this is a non trivial issue for three different reasons: first, for $\bar r =0$ the system is singular (this being possible, of course, only if $R_1 = 0$); second, the system is not Lipschitz continuous when $v=0$ for $p> 2$ (since $\varphi_p^{-1}$ is not Lipschitz at zero); third, the system is not Liptschiz continuous when $u = 0$ for $1 < p < 2$ (since $\hat f$ is not Lipschitz at zero). This subtle problem has been extensively investigated in \cite{ReiWal97}; according to Theorem 4 therein, we can conclude that the uniqueness holds true in each of the following cases:
	\begin{itemize}
		\item $\bar r \geq 0$, $\bar u = 0$ and $\bar v = 0$: this follows from $(f_0)$, corresponding to case $(\delta)$ of \cite[Theorem 4]{ReiWal97};
		\item $\bar r \geq 0$, $\bar v = 0$ and $\bar u \not\in \{0,1\}$ for $p > 2$: this follows from the facts that $\hat f(u) \neq 0$ for $0< u \neq 1$, corresponding to case $(\beta)(v)$ of \cite[Theorem 4]{ReiWal97}, and that $\hat f(u) \equiv 0$ for $u < 0$, since in this can the equation can be explicitly solved; 
		\item $\bar u = 0$ and $\bar v \neq 0$ (hence, $\bar r > 0$) for $1 < p < 2$: this follows again from $(f_0)$, corresponding to case $(\alpha)(iii)$ of \cite[Theorem 4]{ReiWal97}.
	\end{itemize} 
	Thus, the only remaining possibility to be analyzed is $\bar r \geq 0$, $\bar u = 1$ and $\bar v = 0$; in this case, we must show that the only solution
	is $u \equiv 1$. To this end, we define the function 
	$$
	H(r) := \frac{\vert u'(r) \vert^{p}}{p'}+ \hat F(u(r)), 
	$$
	with $r$ in a neighborhood of $\bar{r}$, $\hat F(u) = \int_1^u \hat f(s)\,ds$.
	Notice that, in view
	of $(f_\textrm{eq})$ and of the definition of $\hat f$, it holds that $\hat F(s) \geq 0$ for any $s \in \mathbb{R}$ and $\hat F(s) = 0$ if and only if $s=1$. Hence $H(r) \geq 0$ and $H(r) = 0$ if and only if $u(r) = 1$ and $u'(r) = 0$. In particular, $H(\bar r)=0$. Observing that 
	$$
	\vert u'(r) \vert^p = \vert \varphi_p(u'(r)) \vert^{p'}
	$$
	and that
	$$	
	\left(\varphi_p(u'(r))\right)' = - \frac{N-1}{r} \varphi_p(u'(r)) - \hat f(u(r)) \quad \mbox{ for } r \neq 0,
	$$	
	a straightforward computation yields
	$$
	H'(r) = - \frac{N-1}{r}\vert u'(r) \vert^p\le0 \quad \mbox{ for } r \neq 0.
	$$
	It follows that $H(r) = 0$ for $r \geq \bar{r}$, so that $u(r)= u(\bar r)= 1$ for $r \geq \bar{r}\ge 0$. If $\bar r = 0$, this is enough to conclude; if $\bar{r} > 0$ we also need to check the backward uniqueness. To this end, we observe that 
	$$
	\vert H'(r) \vert = \frac{N-1}{r} \vert u'(r) \vert^p \leq \eta H(r)
	$$
	for $r > 0$ in a neighborhood of $\bar r$ and $\eta > 0$ a suitable constant (depending on the neighborhood). Hence, by Gronwall's Lemma,
	$$
	H(r) \leq H(\bar r) e^{\eta  \vert r-\bar r \vert}
	$$
	for $r$ in a (left) neighborhood of $\bar r$. Again, this implies $H(r) = 0$ and, finally, $u(r) = 1$ in a (left) neighborhood of $\bar r$.
	\smallbreak
	We now prove that the solution $(u_d,v_d)$ can be globally extended to the whole interval $[R_1,R_2]$. By contradiction, suppose that its maximal interval of definition is $[R_1,r^*)$ for some $r^* \leq R_2$; then, the standard theory of ODEs implies that
	\begin{equation}\label{blow}
	\lim_{r \to (r^*)^-} \left(\vert u_d(r) \vert + \vert v_d(r) \vert\right) = +\infty.
	\end{equation}
Since $\hat F\ge0$	and $H'(r)\le0$, we get 
	$$
	\frac{|u_d'(r)|^p}{p'}\le H(r)\le H(R_1)\quad\mbox{for all }r\in[R_1,r^*),
	$$
that is $|u_d'|$ is bounded. Consequently, 
	$$
	|v_d(r)|=r^{N-1}|u_d'(r)|^{p-1}\le C\quad\mbox{and}\quad 
	|u_d(r)|\le u_d(R_1)+\int_{R_1}^r|u_d'(s)|ds\le C'
	$$
	for all $r\in[R_1,r^*)$ and for some $C$, $C'>0$. 
	Hence, \eqref{blow} cannot occur, and so $(u_d,v_d)$ can be extended to the whole interval $[R_1,R_2]$. 
	\smallbreak
	Finally, having proved the uniqueness and global continuability, the continuous dependence property follows from the standard theory of ODEs (compare again with \cite{ReiWal97}).
\end{proof}

Notice now that, as a consequence of the uniqueness of the solutions to the Cauchy problems proved in Lemma~\ref{uni}, we have that, if $d \in (0,1]$,
$$
(u_d(r),u'_d(r)) \neq (1,0) \quad \mbox{ for every } r \in [R_1,R_2].
$$
Accordingly, we can investigate the behavior of the solution $(u_d,v_d)$ to \eqref{sys1}-\eqref{ci}, by introducing a system of polar-like coordinates around the point $(1,0)$. 
Precisely, we set
\begin{equation}\label{polari}
\begin{cases}
u(r)-1=\rho(r)^\frac{2}{p} \cos_p(\theta(r))\\
v(r)=-\rho(r)^\frac{2}{p'} \sin_p(\theta(r)),
\end{cases}
\end{equation}
where $(\cos_p,\sin_p)$ is the unique solution of
$$
\begin{cases}
x'=-\varphi_{p'}(y),\\
y'=\varphi_p(x),\\
x(0)=1,\;\;y(0)=0.
\end{cases}
$$
These functions were first introduced in \cite{DelPinoElguetaManasevich89} (see also \cite{DelPinoManasevichMurua92}, \cite{FabryFayyad92}). They are called $p$-cosine and $p$-sine functions because they share many properties with the classic cosine and sine functions, as we recall below.

\begin{lemma}[{\cite[Lemma 2.1]{YanZhang10}}] \label{propCpSp}
Let $\pi_p:=\frac{2\pi(p-1)^{1/p}}{p\sin(\pi/p)}$, then
	\begin{enumerate}
		\item[$(i)$] both $\cos_p(\theta)$ and $\sin_p(\theta)$ are $2\pi_p$-periodic;
		\item[$(ii)$] $\cos_p$ is even in $\theta$ and $\sin_p$ is odd in $\theta$;
		\item[$(iii)$] $\cos_p(\theta+\pi_p)=-\cos_p(\theta)$, $\sin_p(\theta+\pi_p)=-\sin_p(\theta)$;
		\item[$(iv)$] $\cos_p(\theta)=0$ if and only if $\theta=\pi_p/2+k\pi_p$, $k\in \Z$, and $\sin_p(\theta)=0$ if and only if $\theta=k\pi_p$, $k\in \Z$;
		\item[$(v)$] $\frac{\textrm{d}}{\textrm{d}\theta}\cos_p(\theta)=-\varphi_{p'}(\sin_p(\theta))$ and $\frac{\textrm{d}}{\textrm{d}\theta}\sin_p(\theta)=\varphi_p(\cos_p(\theta))$;
		\item[$(vi)$] $|\cos_p(\theta)|^p/p+|\sin_p(\theta)|^{p'}/p'\equiv 1/p$.
	\end{enumerate}
\end{lemma}

Via the change of coordinates \eqref{polari}, 
system \eqref{sys1} is transformed into
\begin{equation}\label{sys2}
\left\{
\begin{array}{l}
\displaystyle \rho'(r)=\frac{p}{2\rho(r)}\, u'(r) \, \left[\varphi_p(u(r)-1)-r^{(N-1)p'} \hat f(u(r)) \right] \vspace{0.2cm}  \\
\displaystyle \theta'(r)=\frac{r^{N-1}}{\rho^2(r)}\left[ (p-1)|u'(r)|^p+(u(r)-1) \hat f(u(r)) \right].
\end{array}
\right.
\end{equation}
Moreover, we can write the initial condition \eqref{ci} as
\begin{equation}\label{ci2}
\rho(R_1) = d^{\frac{p}2}, \qquad \theta(R_1) = \pi_p,
\end{equation}
and denote the corresponding solution by $(\rho_d,\theta_d)$. It is then easy to realize that the couple $(\rho_d,\theta_d)$ gives rise to a solution of \eqref{ode2} (and in turn of \eqref{main}) if and only if $\theta_d(R_2)=k\pi_p$ for some $k\in\Z$.

We conclude this preliminary section by noting for further convenience that, as an immediate consequence of Lemma \ref{uni}, \eqref{polari} and Lemma \ref{propCpSp}, we have the following.

\begin{corollary}\label{cor:limit_rho_as_d_to0}
If $(d_n)\subset(0,1)$ is such that $d_n \to d \in (0,1]$, then $(\rho_{d_n}(r),\theta_{d_n}(r)) \to (\rho_d(r),\theta_d(r))$ uniformly in $r \in [R_1,R_2]$. Furthermore, 
\begin{equation}\label{2.14}
\lim_{d\to0} \sup_{r\in [R_1,R_2]} \rho_d(r)=0.
\end{equation}
\end{corollary}

\begin{remark}\label{cambiovariabile}
It is worth noticing that, when $R_1 > 0$ (that is, if $\Omega$ is an annulus), we can perform a change of variables which transfors the equation appearing in \eqref{ode2} into a simpler one. Precisely, for $r\in (R_1,R_2)$,  let
$$
t(r):=\int_{R_1}^r s^{-\frac{N-1}{p-1}} \,ds
=\begin{cases}
\frac{p-1}{p-N}\left(r^\frac{p-N}{p-1}-R_1^\frac{p-N}{p-1}\right) \quad&\text{if } p\neq N\\
\ln\frac{r}{R_1}&\text{if } p= N;
\end{cases}
$$
then, $t(r)$ is invertible with inverse
$$
r(t)=\begin{cases}
\left(\frac{p-N}{p-1}t+R_1^\frac{p-N}{p-1}\right)^\frac{p-1}{p-N} \quad&\text{if } p\neq N\\
R_1e^t &\text{if } p= N.
\end{cases}
$$
Setting
$$
w(t):=u(r(t)), \quad T:=t(R_2), \quad a(t):=r(t)^\frac{p(N-1)}{p-1},
$$
we find that \eqref{ode2} is equivalent to
$$
\begin{cases}
-(\varphi_p(w'))'=a(t)\hat f(w) \quad & \mbox{in } (0,T) \\
w'(0)=w'(T)=0.
\end{cases}
$$
The same procedure can be used in the case $R_1 = 0$ (namely, $\Omega$ is a ball) with $p > N$, since also in this case $r(t)$, and consequently $a(t)$, is well defined for all $t\in[0,T]$. 
However, here we prefer to work always with the boundary value problem \eqref{ode2} in order to produce a common proof for all our results. \end{remark}

\subsection{The associated eigenvalue problem}\label{A}
Consider the eigenvalue problem 
\begin{equation}\label{eq:eigenvalue-Ndim}
\begin{cases}
-\Delta_p \phi = \lambda|\phi|^{p-2}\phi\quad&\mbox{in }\Omega\\
\partial_\nu \phi=0&\mbox{on }\partial\Omega,
\end{cases}
\end{equation}
where $\Omega$ is one of the two radial open domains defined in the introduction, and $\lambda \in \mathbb{R}$. Since we are interested only in the radial eigenvalues of \eqref{eq:eigenvalue-Ndim}, we can rewrite \eqref{eq:eigenvalue-Ndim} as the following 1-dimensional eigenvalue problem
\begin{equation}\label{eq:eigenvalue}
\begin{cases}
-(r^{N-1}\varphi_p(\phi'))'=\lambda r^{N-1} \varphi_p(\phi) \quad & \mbox{in } (R_1,R_2) \\
\phi'(R_1)=\phi'(R_2)=0.
\end{cases}
\end{equation}
 
The following result is well-known. 

\begin{theorem}[Theorem 1 of \cite{RW99}]\label{th:eigen}
The eigenvalue problem \eqref{eq:eigenvalue} has a countable number of simple eigenvalues $0=\lambda_1<\lambda_2<\lambda_3<\dots$, $\lim_{k\to\infty}\lambda_k=+\infty$, and no other eigenvalues. The eigenfunction $\phi_k$ that corresponds to the $k$-th eigenvalue $\lambda_k$ has $k-1$ simple zeros in $(R_1,R_2)$.  
\end{theorem}

We remark that for every $1\le k\in\mathbb N$, if we denote by $\lambda_k$ the $k$-th eigenvalue of \eqref{eq:eigenvalue} and by $\lambda_{k}^{\mathrm{rad}}$  the $k$-th radial eigenvalue of \eqref{eq:eigenvalue-Ndim},
$$
\lambda_k=\lambda_k^{\mathrm{rad}}.
$$ 
Following Sturm's theory, we are now going to clarify the relationship between the eigenvalues $\lambda_k^{\mathrm{rad}}$ and an angular coordinate $\vartheta$ analogous to the one defined in the previous section. Accordingly, we consider the change of variables 
\begin{equation}\label{polar-eigenprobl}
\begin{cases}
\phi(r)=\varrho_\lambda(r)^{\frac{2}{p}}\cos_p(\vartheta_\lambda(r))\\
r^{N-1}\varphi_p(\phi'(r))=-\varrho_\lambda(r)^{\frac{2}{p'}}\sin_p(\vartheta_\lambda(r)),
\end{cases}
\end{equation}
where the functions $\sin_p$, $\cos_p$ are defined as in the previous section. Then, the eigenvalue problem \eqref{eq:eigenvalue} reads as 
$$
\begin{cases}
\displaystyle \varrho_\lambda'(r)=\frac{p}{2\varrho(r)}\left(1-\lambda r^{(N-1)p'}\right)\varphi_p(\phi(r))\phi'(r),
\vspace{0.2cm}\\
\displaystyle \vartheta_\lambda'(r)=\frac{r^{N-1}}{\varrho(r)^2}\left[(p-1)|\phi'(r)|^p+\lambda |\phi(r)|^p\right],
\end{cases}
$$
with boundary conditions 
$$
\vartheta_\lambda(R_1)=\pi_p\quad\mbox{and}\quad \vartheta_\lambda(R_2)=j\pi_p
$$ 
for some $j\in\mathbb N$. Notice that the function $r \mapsto \vartheta_\lambda(r)$ is strictly increasing. As a consequence, if $\lambda=\lambda_k$ for $k\ge1$, the fact that $\phi_k$ has $k-1$ simple zeros in $(R_1,R_2)$ reads as 
\begin{equation}\label{initialcondition-ep}
\vartheta_{\lambda_k}(R_1)=\pi_p\quad\mbox{and}\quad \vartheta_{\lambda_k}(R_2)=k\pi_p.
\end{equation}
For further convenience, we also observe that, by \eqref{polar-eigenprobl}, 
$$
r^{(N-1)p'}|\phi'|^p=\varrho_\lambda^2|\sin_p(\vartheta_\lambda)|^{p'}\quad\mbox{ and }\quad|\phi|^p=\varrho_\lambda^2|\cos_p(\vartheta_\lambda)|^p,
$$
so that
\begin{equation}\label{eq:theta'-omogenea-associata}
\vartheta_\lambda' =r^{N-1}\left[\frac{p-1}{r^{(N-1)p'}}|\sin_p(\vartheta_\lambda)|^{p'}+\lambda|\cos_p(\vartheta_\lambda)|^p\right].
\end{equation}

\subsection{The proof of Theorem \ref{th:main}} \label{sec:theo1}
\begin{proof}
By $(f_1)$, we know that for all $n\in\mathbb N$ there exists $\delta=\delta(n)>0$ such that for every $s$ satisfying $|s-1|<\delta$ it holds
$$
\hat f(s)(s-1)=f(s)(s-1)>
\begin{cases}
\left(C_1-\dfrac{1}{n}\right)|s-1|^p&\mbox{ if }C_1\in(0,\infty),
\vspace{0.1cm}\\
n|s-1|^p&\mbox{if } C_1=+\infty.
\end{cases}
$$
Then, by \eqref{sys2}, we get that if $|u(r)-1|<\delta(n)$
\begin{equation}\label{eq:disug-theta'}
\theta'(r)>
\begin{cases}
\displaystyle\frac{r^{N-1}}{\rho(r)^2}\left[(p-1)|u'(r)|^p+\left(C_1-\frac1n\right)|u(r)-1|^p\right] 
&\mbox{ if }C_1\in(0,\infty),
\vspace{0.2cm} \\
\displaystyle \frac{r^{N-1}}{\rho(r)^2}\left[(p-1)|u'(r)|^p+n|u(r)-1|^p\right] 
\quad&\mbox{if }C_1=+\infty.
\end{cases}
\end{equation}
Furthermore, we deduce from \eqref{polari} that
$$
r^{(N-1)p'}|u'|^p=\rho^2|\sin_p(\theta)|^{p'} \quad\text{and}\quad |u-1|^p=\rho^2|\cos_p(\theta)|^p.
$$
Hence, combining the latter equalities with \eqref{eq:disug-theta'}, we obtain by \eqref{2.14} that for every $n$ there exists $\delta'=\delta'(n)>0$ such that for all $d\in (0,\delta')$
\begin{equation}\label{eq:disug-theta'1}
\theta_d'(r)>
\begin{cases}
\displaystyle  r^{N-1}\left[\frac{p-1}{r^{(N-1)p'}}|\sin_p(\theta_d(r))|^{p'}+\!\left(\!C_1\!-\!\frac1n\!\right)\!|\cos_p(\theta_d(r))|^p\right]&\mbox{if }C_1\in(0,\infty),
\vspace{0.2cm}\\
\displaystyle r^{N-1}\left[\frac{p-1}{r^{(N-1)p'}}|\sin_p(\theta_d(r))|^{p'}+n|\cos_p(\theta_d(r))|^p\right]&\mbox{if }C_1=+\infty
\end{cases}
\end{equation}
for all $r\in[R_1,R_2]$.

Now, in both cases (i.e., $0<C_1<+\infty$ and $C_1=+\infty$), let $k\ge1$ be an integer such that $C_1>\lambda_{k+1}^{\mathrm{rad}}$. For $n$ so large that
$$
n>\lambda_{k+1}\quad\mbox{ and }\quad C_1-\frac1n>\lambda_{k+1},
$$ 
relation \eqref{eq:disug-theta'1} becomes
\[
\theta_d'(r)> r^{N-1}\left[\frac{p-1}{r^{(N-1)p'}}|\sin_p(\theta_d(r))|^{p'}+ \lambda_{k+1} |\cos_p(\theta_d(r))|^p\right]
\]

Hence,
recalling \eqref{eq:theta'-omogenea-associata} with $\lambda=\lambda_{k+1}$ and using the Comparison Theorem for ODEs, we obtain for $d$ small enough
\begin{equation}\label{eq:theta_d_lambda}
\theta_d(r)> \vartheta_{\lambda_{k+1}}(r)\quad\mbox{for all }r\in(R_1,R_2].
\end{equation}
In particular, by \eqref{initialcondition-ep}
$$\theta_d(R_2)> (k+1)\pi_p$$
for $d$ sufficiently close to 0.
Since $\theta_1(R_2)=\pi_p$, by the continuity of the map $d \mapsto \theta_d(R_2)$ (see Corollary \ref{cor:limit_rho_as_d_to0}), we have that for all $j=1,\dots,k$ there exists $d_j\in(0,1)$ for which $\theta_{d_j}(R_2)=(j+1)\pi_p$. This corresponds to $u_{d_j}'(R_2)=0$, providing the desired solution $u_j$ of \eqref{main}.

In order to prove the oscillatory behavior of $u_j$ it suffices to remark that, since $\theta_{d_j}(r)$ is monotone increasing (see \eqref{sys2} and recall $(f_\textrm{eq})$), there exist exactly $j$ radii $r_1,\ldots,r_{j}\in (R_1,R_2)$ such that $\theta_{d_j}(r_1)=\frac{3}{2}\pi_p$, $\theta_{d_j}(r_2)=\frac{5}{2}\pi_p,\ldots,\theta_{d_j}(r_{j})=\left(j+\frac{1}{2}\right)\pi_p$.
\end{proof}

\begin{remark}\label{conv} Let $\vartheta_{C_1}$ be the angular coordinate defined in \eqref{polar-eigenprobl} with $\lambda=C_1$.
In the case $C_1<\infty$, we can show that $\theta_d\to\vartheta_{C_1}$ uniformly in $[R_1,R_2]$ for $d \to 0^+$ (in particular, this holds true also for $C_1=0$), thus obtaining a stronger relation than \eqref{eq:theta_d_lambda}. Indeed, by \eqref{sys2} and $(f_1)$, we have as $d\to0$
$$
\theta_d'=(p-1)r^{(N-1)(1-p')}|\sin_p(\theta_d)|^{p'}+C_1 r^{N-1} |\cos_p(\theta_d)|^p+r^{N-1}\frac{o(\rho_d^2|\cos_p(\theta_d)|^p)}{\rho_d^2}.
$$
Now, Corollary \ref{cor:limit_rho_as_d_to0} and $|\cos_p(\theta_d)|^p=O(1)$ provide $o(\rho_d^2|\cos_p(\theta_d)|^p)/\rho_d^2=o(1)$ as $d\to 0^+$. 
On the other hand, by \eqref{eq:theta'-omogenea-associata},
$$
\vartheta_{C_1}'=(p-1)r^{(N-1)(1-p')}|\sin_p(\vartheta_{C_1})|^{p'}+C_1 r^{N-1} |\cos_p(\vartheta_{C_1})|^p,
$$
whence 
$$
|\theta_d'-\vartheta_{C_1}'|\le \mathcal L|\theta_d-\vartheta_{C_1}|+o(1)\quad\mbox{as }d\to0^+,
$$
with $\mathcal L=\mathcal L(p,C_1,\sin_p,\cos_p)>0$ being related to the Lipschitz constants of $|\sin_p(\cdot)|^{p'}$ and $|\cos_p(\cdot)|^p$. Therefore, for all $\varepsilon>0$ and all $r\in[R_1,R_2]$,
$$
|\theta_d(r)-\vartheta_{C_1}(r)|\le \mathcal L\int_{R_1}^r|\theta_d(s)-\vartheta_{C_1}(s)|ds+\varepsilon \quad\mbox{for $d$ sufficiently small.}
$$
By Gronwall's inequality, for all $r\in[R_1,R_2]$
$$
|\theta_d(r)-\vartheta_{C_1}(r)|\le \varepsilon e^{\mathcal L (r-R_1)}\quad\mbox{for $d$ sufficiently small}
$$
and so, by the arbitrariness of $\varepsilon>0$, 
$$
|\theta_d-\vartheta_{C_1}| \to 0 \quad\mbox{uniformly in }[R_1,R_2] \mbox{ as } d \to 0^+.
$$
\end{remark}

\subsection{The proof of Theorem \ref{th:main2}}\label{sec:theo2}
We will adapt an argument introduced in \cite{BosZan13}, and make use of the following result proved therein.

\begin{lemma}[{\cite[Corollary 5.1]{BosZan13}}]\label{lem:BosZan}
Let us consider the system
\begin{equation}\label{sys4}
\left\{
\begin{array}{l}
\rho'= \mathrm{P}(r,\rho,\theta)  \\
\displaystyle \theta'=\Theta(r,\rho,\theta),
\end{array}
\right.
\end{equation}
being $\mathrm{P},\Theta:[r_1,r_2]\times\R_0^+\times[\theta_1,\theta_2]\to\R$ continuous functions. Suppose that the uniqueness for the Cauchy problem associated with \eqref{sys4} is ensured and let $\gamma:[\theta_1,\theta_2]\to\R$ be a function of class $C^1$, with $\gamma(\theta)>0$ for every $\theta\in [\theta_1,\theta_2]$. Assume
$$
\mathrm{P}(r,\gamma(\theta),\theta) \leq \gamma'(\theta)\Theta(r,\gamma(\theta),\theta) 
\quad \text{for every } r\in[r_1,r_2],\ \theta\in [\theta_1,\theta_2].
$$
Then for every $(\rho,\theta):I\to \R_0^+\times[\theta_1,\theta_2]$ solution to \eqref{sys4} (being $I\subset [r_1,r_2]$ an interval) and $r_0\in I$,
\[
\rho(r_0)\leq \gamma(\theta(r_0)) \quad \Longrightarrow\quad \rho(r)\leq\gamma(\theta(r))\text{ for every } r\in (r_0,+\infty)\cap I.
\]
\end{lemma}

For the proof of Theorem \ref{th:main2}, it is convenient to write the equation in \eqref{ode2} as the planar system in $(R_1,R_2)$
\begin{equation}\label{sys3}
u' = \varphi_p^{-1}\left( \left(\frac{R_2}{r}\right)^{N-1} v\right), \qquad
v' = - \left(\frac{r}{R_2}\right)^{N-1} \hat f(u).
\end{equation}
The advantage of this new scaling is that the maximum of the weight $(r/R_2)^{N-1}$ in $[R_1,R_2]$ is independent of $R_2$, a property that will be useful in the sequel. While, concerning the minimum of the same weight, we will use the fact that it is positive in $[\varepsilon R_2,R_2]$ for any $\varepsilon>0$.
Comparing \eqref{sys3} with \eqref{sys1}, it is immediately realized that all the properties discussed in Section \ref{preliminary} still hold true for this slightly different planar formulation of \eqref{ode2}. In particular, we define  $(u_d,v_d)$ as the solution of \eqref{sys3} satisfying $(u_d(R_1),v_d(R_1)) = (1-d,0)$ and we pass to polar-like coordinates around the point $(1,0)$ as in \eqref{polari}, that is,
$$
\begin{cases}
x(r) := u(r)-1=\rho(r)^\frac{2}{p} \cos_p(\theta(r))\\
y(r) := v(r)=-\rho(r)^\frac{2}{p'} \sin_p(\theta(r)).
\end{cases}
$$
We thus obtain (compare with \eqref{sys2}) the system
\begin{equation}\label{sys4a}
\left\{
\begin{array}{l}
\displaystyle \rho'= \frac{p}{2\rho}\, \left(\frac{R_2}{r} \right)^{(N-1)(p'-1)} \varphi_{p'}(y) \, \left[\varphi_p(x)-\left(\frac{r}{R_2}\right)^{(N-1)p'} \hat f(x+1) \right] =: \mathrm{P}(r,\rho,\theta)
  \\
\displaystyle \theta'= \frac{1}{\rho^2} \left(\frac{R_2}{r} \right)^{(N-1)(p'-1)}\left[ (p-1) \vert y \vert^{p'}+\left(\frac{r}{R_2}\right)^{(N-1)p'} \hat f(x+1)x \right] =: \Theta(r,\rho,\theta),
\end{array}
\right.
\end{equation}
with initial conditions \eqref{ci2}.
We also write
$$
\mathrm{P}(r,\rho,\theta) =: \rho S\left(r,\rho^{\tfrac{2}{p}}\cos_p(\theta),
	-\rho^{\tfrac{2}{p'}}\sin_p(\theta)\right)
$$
and
$$
\Theta(r,\rho,\theta) =: U\left(r,\rho^{\tfrac{2}{p}}\cos_p(\theta),
-\rho^{\tfrac{2}{p'}}\sin_p(\theta)\right),
$$
where, noting that $\vert x \vert^p + (p-1) \vert y \vert^{p'}=\rho^2$,
$$
S(r,x,y) = \frac{p}{2}\, \left(\frac{R_2}{r} \right)^{(N-1)(p'-1)} \cdot \frac{ \varphi_{p'}(y) \, \left[\varphi_p(x)-\left(\dfrac{r}{R_2}\right)^{(N-1)p'} \hat f(x+1) \right]}{\vert x \vert^p + (p-1) \vert y \vert^{p'}}
$$
and
$$
U(r,x,y) = \left(\frac{R_2}{r} \right)^{(N-1)(p'-1)} \cdot \frac{\left[ (p-1) \vert y \vert^{p'}+\left(\dfrac{r}{R_2}\right)^{(N-1)p'} \hat f(x+1)x \right]}{\vert x \vert^p + (p-1) \vert y \vert^{p'}}.
$$

\begin{proof}[$\bullet$ Proof of Theorem \ref{th:main2}] We treat the two cases $\Omega=\mathcal B(R_2)$ and $\Omega=\mathcal A(R_1,R_2)$ simultaneously, by taking into account that the condition $R_1<\varepsilon R_2$ is trivially verified for all $\varepsilon>0$ when $R_1=0$, that is in the case of the ball. Hence, if $\Omega=\mathcal B(R_2)$, for any $k\ge1$ we can fix any $\varepsilon>0$ and consider $R_*$ only depending on $k$.  

For $d\in [0,1]$ let $(\rho_d,\theta_d)$ be the solution of \eqref{sys4a} with initial conditions \eqref{ci2}.
The key point is to show that for any integer $k \geq 1$ and any $\varepsilon > 0$,
there exists $R_*(k,\varepsilon) > 0$ such that for $R_1 < \varepsilon R_2$ and $R_2 > R_*(k,\varepsilon)$ there exists $d_k \in (0,1)$ such that 
$\theta_{d_k}(R_2) > (k+1)\pi_p$. From this, one can easily conclude. Indeed, on one hand $\theta_1(R_2) = \pi_p$. On the other hand, $\theta_d(R_2) < 2\pi_p$ for $d$ small enough, since by Remark \ref{conv} it holds $\theta_d(R_2) \to \vartheta_0(R_2)$ for $d \to 0$ and $\vartheta_0(r) \equiv \pi_p$. Then, by continuity, it is possible to find for any $j=1,\ldots,k$ two values 
$$
0 < d_j^- < d_k < d_j^+ < 1 
$$ 
such that $\theta_{d_j^\pm}(R_2) = (j+1)\pi_p$, giving rise to the desired solutions $u^\pm_j$. The oscillatory behavior is then proved as in Theorem \ref{main}. In fact, by \eqref{sys4a} $\theta_{d_j^\pm}$ is increasing for every $j=1,\dots, k$, and consequently, there exist exactly $2j$ radii  $r_1^-,\ldots,r^-_{j}, r_1^+,\ldots,r^+_{j}\in (R_1,R_2)$ such that $\theta_{d^\pm_j}(r^\pm_1)=\frac{3}{2}\pi_p$, $\theta_{d^\pm_j}(r^\pm_2)=\frac{5}{2}\pi_p$,$\dots,$ $\theta_{d^\pm_j}(r^\pm_{j})=\left(j+\frac{1}{2}\right)\pi_p$.

From now on, we thus focus on the proof of the above claim; this requires, however, several auxiliary definitions. First of all, we set
$$
\mathcal{M}_-(x,y) := \begin{cases}
\displaystyle{\frac{p}{2}\cdot \frac{\varphi_{p'}(y)\left( \varphi_p(x) - \hat f(x+1)\right)}{(p-1)\vert y \vert^{p'} + \hat f(x+1)x}}  &\text{if } xy \geq 0, \vspace{0.2cm} \\
\displaystyle{\frac{p}{2}\cdot \frac{\varphi_{p'}(y)\left( \varphi_p(x) - \varepsilon^{(N-1)p'} \hat f(x+1)\right)}{(p-1)\vert y \vert^{p'} + \varepsilon^{(N-1)p'} \hat f(x+1)x}} &\text{if } xy \leq 0,
\end{cases} 
$$ 
and
$$
\mathcal{M}_+(x,y) := \begin{cases}
\displaystyle{\frac{p}{2}\cdot \frac{\varphi_{p'}(y)\left( \varphi_p(x) - \varepsilon^{(N-1)p'} \hat f(x+1)\right)}{(p-1)\vert y \vert^{p'} + \varepsilon^{(N-1)p'} \hat f(x+1)x}} &\text{if } xy \geq 0, \vspace{0.2cm}\\
\displaystyle{\frac{p}{2}\cdot \frac{\varphi_{p'}(y)\left( \varphi_p(x) - \hat f(x+1)\right)}{(p-1)\vert y \vert^{p'} + \hat f(x+1)x}} &\text{if } xy \leq 0.
\end{cases} 
$$ 
A straightforward calculation shows that 
\begin{equation}\label{msu}
\mathcal{M}_-(x,y) \leq \frac{S(r,x,y)}{U(r,x,y)} \leq \mathcal{M}_+(x,y)
\end{equation}
for all $r \in [\varepsilon R_2,R_2]$ and all $(x,y) \in \mathbb{R}^2 \setminus \{(0,0)\}$. For example, the inequality $\mathcal{M}_- \leq S/U$ is equivalent to
\[
\varphi_{p'}(y) \hat f(x+1) \rho^2 \left(\frac{r}{R_2}\right)^{(N-1)p'} \leq \varphi_{p'}(y) \hat f(x+1) \rho^2, \quad \text{if } xy\geq0,
\]
 and to
\[
\varphi_{p'}(y) \hat f(x+1) \rho^2 \left(\frac{r}{R_2}\right)^{(N-1)p'} \leq \varphi_{p'}(y) \hat f(x+1) \rho^2 \varepsilon^{(N-1)p'}, \quad \text{if } xy\leq0,
\]
from which we see that the first inequality in \eqref{msu} is satisfied for $r \in [\varepsilon R_2,R_2], (x,y) \in \mathbb{R}^2 \setminus \{(0,0)\}$. The proof of the second inequality in \eqref{msu} is similar.

Then, we define $\rho_\pm(\theta;\bar\theta,\bar\rho)$ as the solution of 
\begin{equation}\label{eqm}
\begin{cases}
\frac{d\rho}{d\theta} = \rho \mathcal{M}_\pm \left(\rho^{2/p}\cos_p(\theta),-\rho^{2/p'}\sin_p(\theta)\right)\\
\rho_\pm(\bar\theta;\bar\rho,\bar\theta) = \bar\rho
\end{cases}
\end{equation}
and we set
for any $\bar\rho > 0$
$$
m_k(\bar\rho) := \inf_{\bar\theta \in [0,2\pi_p),\;\theta \in [\bar\theta, \bar\theta + k \pi_p]} \rho_-(\theta;\bar\rho,\bar\theta),
$$
$$
M_k(\bar\rho) := \sup_{\bar\theta \in [0,2\pi_p), \; \theta \in [\bar\theta, \bar\theta + k \pi_p]} \rho_+(\theta;\bar\rho,\bar\theta).
$$
By continuous dependence, we can choose $0 < \check{\rho}_k < \rho^*_k < \hat \rho_k$ such that
\begin{equation}\label{choice}
0 < \check{\rho}_k < m_k(\rho^*_k) \leq \rho_k^* \leq M_k(\rho_k^*) < \hat \rho_k  < 1.
\end{equation}
Finally, we set
$$
\delta_k^* := \inf_{\check{\rho}_k \leq \rho \leq \hat \rho_k} 
\frac{\varepsilon^{N-1} \hat f(x+1)x + (p-1)\vert y \vert^{p'}}{\vert x \vert^p + (p-1)\vert y \vert^{p'}}.
$$
We are now in a position to prove that, if $R_1 < \varepsilon R_2$ and
$$
R_2 > R_*(k,\varepsilon) := \frac{\pi_p k}{(1-\varepsilon)\delta_k^*},
$$
then our claim holds true, namely there exists $d_k \in (0,1)$ such that 
$$
\theta_{d_k}(R_2) > (k+1)\pi_p.
$$ 

We first observe that, since $\rho_1(r) = 0$ and $\rho_0(r) = 1$ for any $r \in [R_1,R_2]$, there exists $d_k \in (0,1)$ such that
$$
\rho_{d_k}(\varepsilon R_2) = \rho_k^*,
$$
reasoning as in Corollary \ref{cor:limit_rho_as_d_to0}.
We are now going to show that 
\begin{equation}\label{finalclaim}
\theta_{d_k}(R_2) - \theta_{d_k}(\varepsilon R_2) > k \pi_p,
\end{equation}
which concludes the proof since $\theta_{d_k}(R_1) = \pi_p$ and, by \eqref{sys4a}, $\theta_{d_k}$ is a non-decreasing function.
We distinguish two cases. If $\rho_{d_k}(r) \in [\check{\rho}_k,\hat \rho_k]$ for any $r \in [\varepsilon R_2, R_2]$, we easily conclude: indeed, by the expression of $\theta'$ in \eqref{sys4a}, the definition of $\delta_k^*$ and the choice of $R_2$,
$$  
\theta_{d_k}(R_2) - \theta_{d_k}(\varepsilon R_2) = \int_{\varepsilon R_2}^{R_2} \theta_{d_k}'(r)\,dr \geq R_2 (1-\varepsilon) 
\delta_k^* > k\pi_p.
$$
Otherwise, we let $\bar r \in [\varepsilon R_2, R_2)$ be the largest value such that $\rho_{d_k}(r) \in [\check{\rho}_k,\hat \rho_k]$ for any
$r \in [\varepsilon R_2, \bar{r}]$ and we prove in this case that
$$
\theta_{d_k}(\bar r) - \theta_{d_k}(\varepsilon R_2) > k\pi_p,
$$
implying \eqref{finalclaim} again in view of the monotonicity of $\theta_{d_k}$.

Suppose by contradiction that this is not true and, just to fix the ideas, that $\rho_{d_k}(\bar r) = \hat \rho_k$ (in the case $\rho_{d_k}(\bar r) = \check \rho_k$ the argument is analogous). Observe also that, again by the monotonicity of 
$\theta_{d_k}$, we have $\theta_{d_k}(r) - \theta_{d_k}(\varepsilon R_2) \leq k \pi_p$ for any $r \in [\varepsilon R_2,\bar r]$.
Now, we consider the function $\gamma(\theta) = \rho^+(\theta;\rho_k^*,\bar\theta)$, where $\bar\theta \in [0,2\pi_p)$ is such that
$\theta_{d_k}(\varepsilon R_2) \equiv \bar\theta \mod 2\pi_p$. By the definition of $M_k(\rho_k^*)$ and \eqref{choice}, it holds
$$
\gamma(\theta) < \hat \rho_k \quad \mbox{ for every } \theta \in [\bar\theta,\bar\theta + k\pi_p];
$$
moreover, from \eqref{msu} and \eqref{eqm} we obtain
$$
\mathrm{P}(r,\gamma(\theta),\theta) \leq \gamma'(\theta)\Theta(r,\gamma(\theta),\theta) \quad \mbox{ for every } r \in 
[\varepsilon R_2,\bar{r}],\, \theta \in [\bar\theta,\bar\theta + k \pi_p].
$$
Lemma \ref{lem:BosZan} then implies that 
$$
\rho_{d_k}(r) \leq \gamma(\theta_{d_k}(r)) \quad \mbox{ for every } r \in [\varepsilon R_2,\bar{r}],
$$
so that $\rho_{d_k}(\bar{r}) \leq M_k(\rho_k^*) < \hat \rho_k$, a contradiction.
\end{proof}

\section{Numerical simulations and open problems}\label{sec:simulations}
We present here some numerical simulations performed with the software AUTO-07p \cite{AUTO}. We consider problem \eqref{modello} in dimension $N=1$, more precisely
\begin{equation}\label{eq:modelloN1}
\begin{cases}
-(\varphi_p(u'))'+u^{p-1}=u^{q-1} \quad &\text{in } (0,1) \\
u>0 \quad &\text{in } (0,1) \\
u'(0)=u'(1)=0.
\end{cases}
\end{equation}

\begin{figure}[h!t]
\begin{center}
\includegraphics[width=\textwidth]{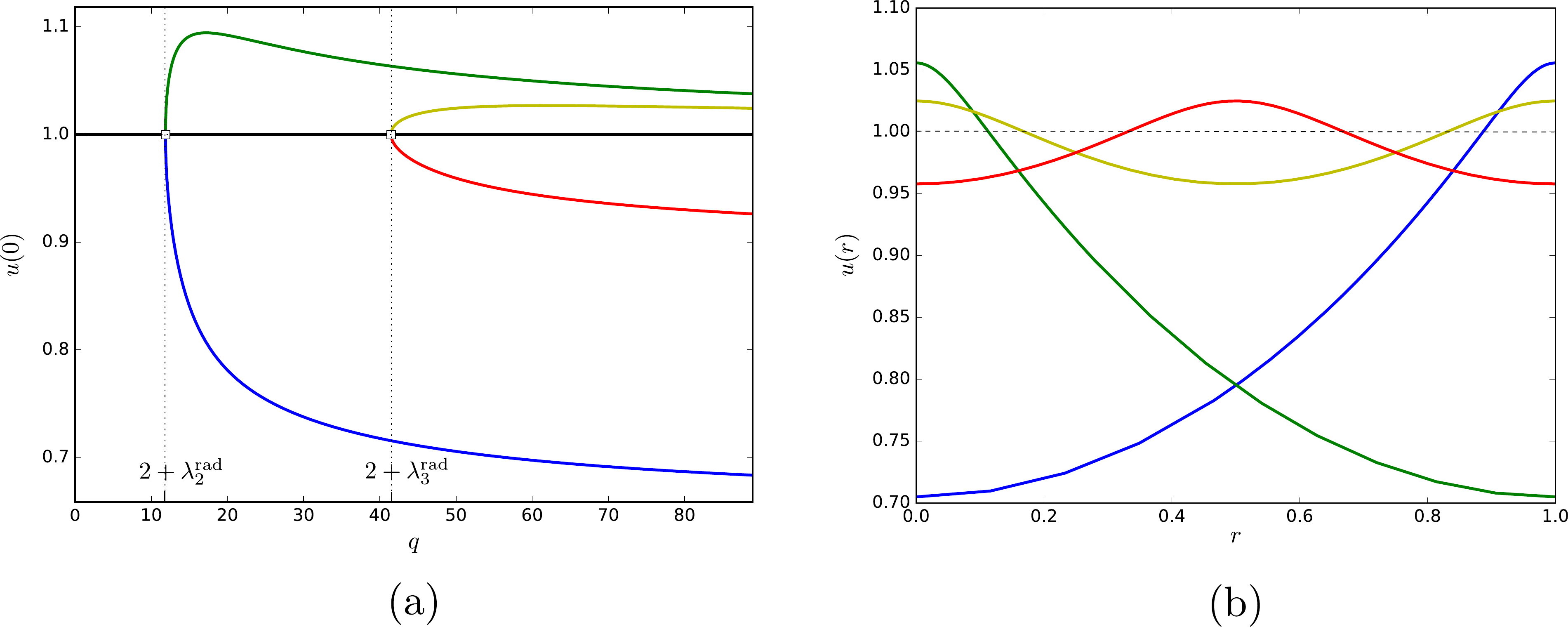}
\caption{The first two bifurcation branches for problem \eqref{eq:modelloN1} in the case $p=2$. The color of each solution in the right plot corresponds to the color of the branch it belongs to in the left plot. (a) Bifurcation diagram: $u(0)$ as function of $q$. (b) Solutions belonging to the first two branches for $q\simeq 50$.}\label{fig:p2}
\end{center}
\end{figure}

In Figure~\ref{fig:p2} we represent the first two bifurcation branches for problem \eqref{eq:modelloN1} in the case $p=2$. The black line represents the constant solution $u\equiv1$; the branches bifurcate at points $q=2+\lambda_k^{\text{rad}}$, $k=2,3$. The solutions belonging to the lower part of the first branch are monotone increasing, the ones belonging to the upper part of the first branch are monotone decreasing, in both cases they all intersect once the constant solution $u\equiv 1$. Solutions of the lower part of the second branch present exactly one  interior maximum point, solutions of the upper part of the second branch have exactly one interior minimum point, in both cases they have two intersections with $u\equiv 1$, and so on. The solutions that we have found in Corollary \ref{cor:modello}-(ii) belong to the lower parts of the branches, since they all satisfy $u(0)<1$. 
Much more general simulations for $p=2$ can be found in \cite{bonheure2016multiple}.
We remark that, as explained therein, the global behavior of the upper parts of the branches can be investigated when the nonlinearity is subcritical (in particular, for $N=1$) or when the problem is considered in an annular domain, while it appears as an open problem in the general setting. We believe that the shooting technique adopted in this paper
could also lead to results similar to the ones obtained in \cite{bonheure2016multiple}, providing 
(in the subcritical setting) multiple positive solutions with $u(0) > 1$. Notice, however, that here we do not obtain bifurcation continua, but just multiple solutions for a fixed value of $q$ (studying their behavior with respect to a parameter is possible in principle, but requires additional arguments from planar topology, see \cite{rebzan}).

\begin{figure}[h!t]
\begin{center}
\includegraphics[width=\textwidth]{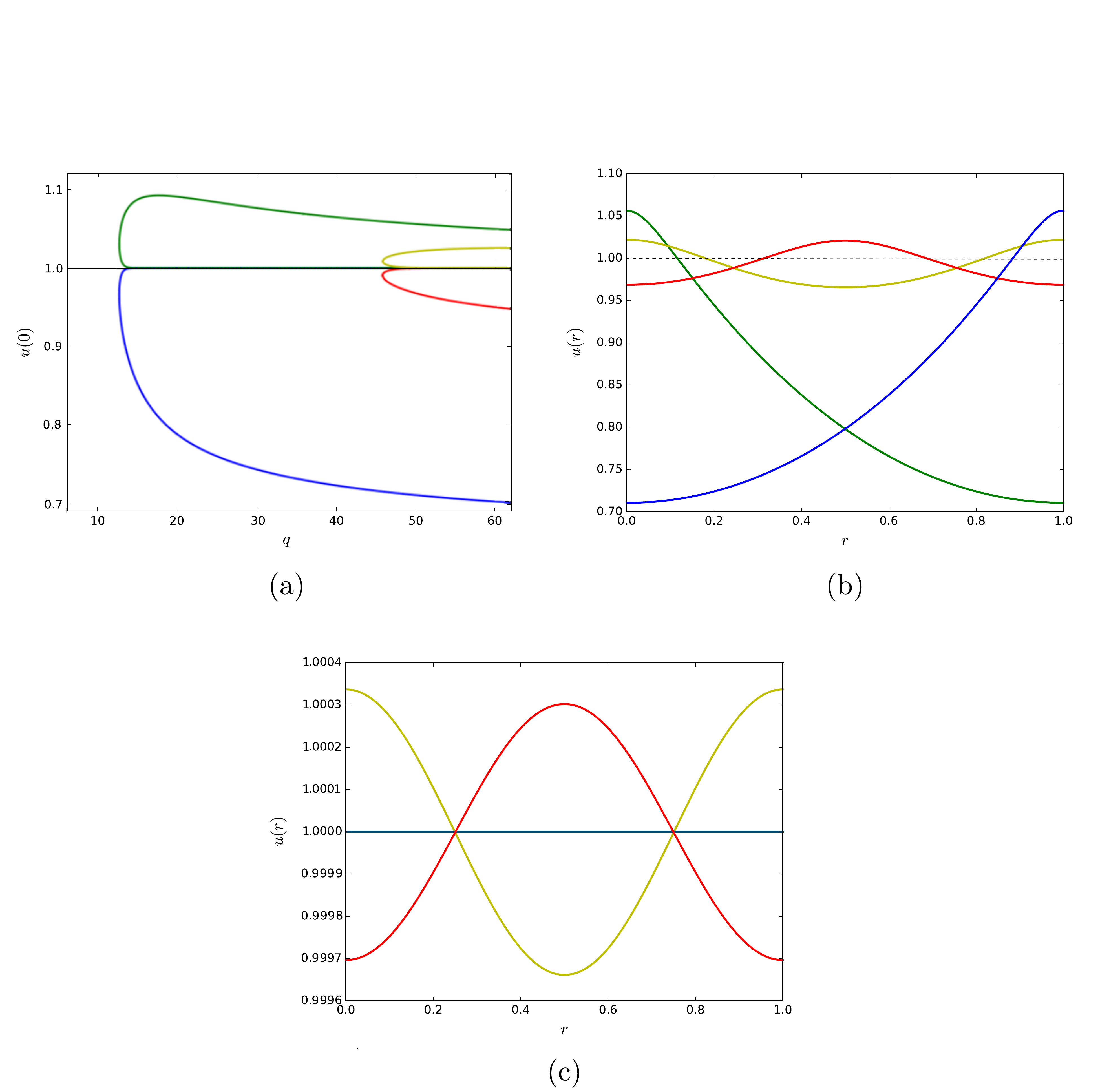}
\caption{The first four ``bifurcation'' branches for problem \eqref{eq:modelloN1} in the case $p=1.97$. The color of each solution in the last two plots  corresponds to the color of the branch it belongs to in the first plot. (a) Bifurcation diagram: $u(0)$ as function of $q$. Notice that the new
folded parts of the branches appear in the figure almost completely
overlapped with the branch of the constant solution. (b) ``Large'' solutions belonging to the first four branches for $q\simeq 50$. (c) ``Almost constant'' solutions belonging to the first four branches, $q\simeq 50$; again, note that the green and the blue solutions appear almost completely overlapped.}\label{fig:pi2}
\end{center}
\end{figure}

In the case $p<2$, the branches persist for $p$ sufficiently close to $2$, but now both the upper and the lower part of each branch split into two.  
In Figure~\ref{fig:pi2} we represent this phenomenon for $p=1.97$. Now we have four branches. According to the simulations, none of them seems to bifurcate from the constant solution: each branch seems to be unbounded on both sides, and one side seems to converge to the constant solution $u\equiv1$ as $q\to+\infty$, as if the bifurcation point had escaped to infinity. For this reason, each branch contains two solutions having the same oscillatory behavior, thus giving rise to the double of solutions with respect to the case $p=2$.
Once again, the solutions that we have found in Corollary \ref{cor:modello}-(iii) belong to the lower branches, since they all satisfy $u(0)<1$. The existence of solutions satisfying $u(0)>1$ for $p<2$ is for the moment an open problem. Similarly as in the case $p=2$, we conjecture that such solutions should exist when $f$ has Sobolev-subcritical growth, or when the domain is an annulus, thus giving rise, in the assumptions of Theorem \ref{th:main2}, to $4k$ radial solutions.

\begin{figure}[h!t]
\begin{center}
\includegraphics[height=4.6cm]{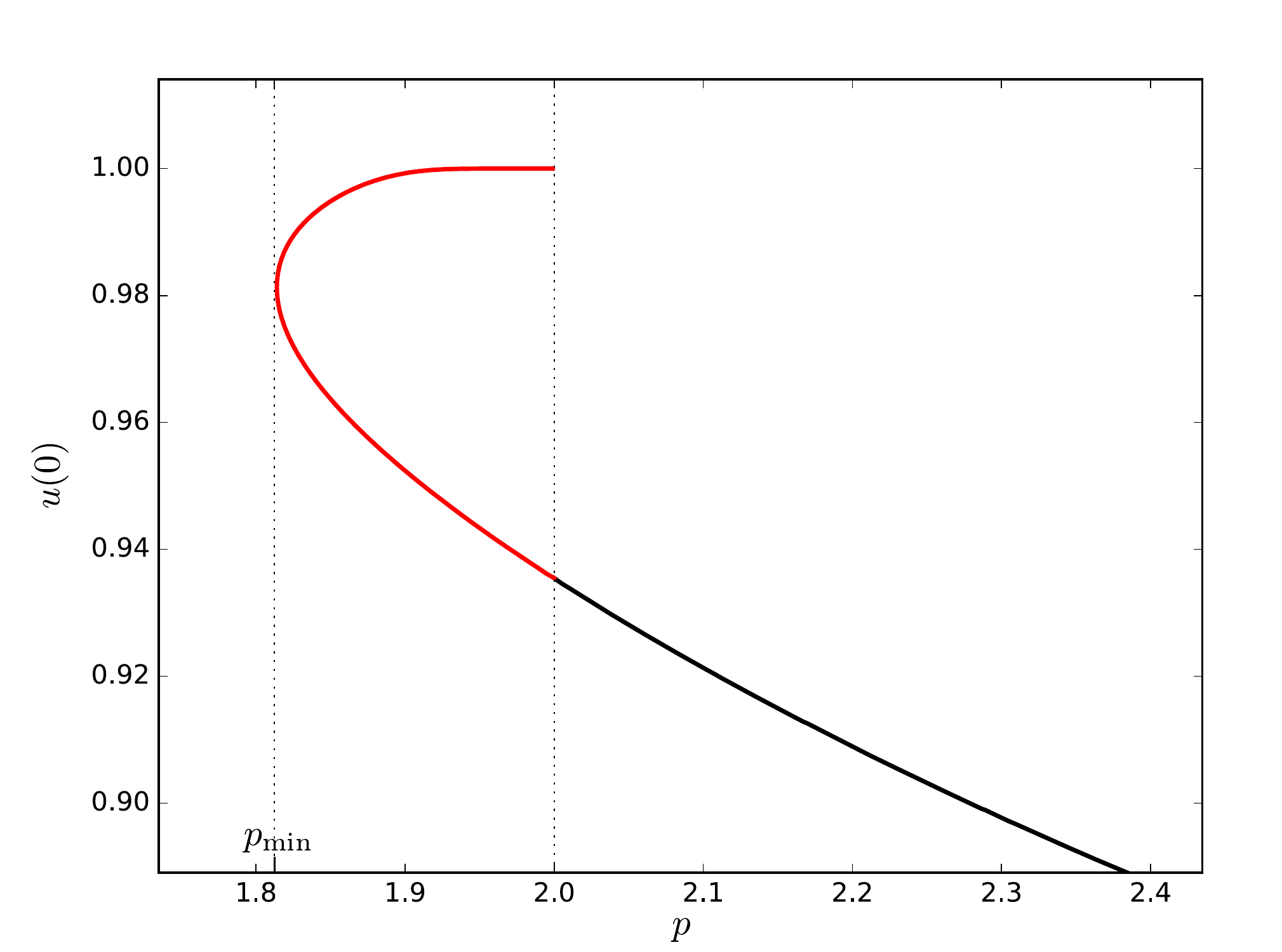} 
\caption{Bifurcation in the parameter $p$, starting from $p=2$ for $q=\bar q$ fixed. In particular, for $p < 2$, a branch of solutions (in red) is obtained; this folded branch persists for $p \geq p_{\text{min}}$, giving rise to a couple of solutions for every $p\in (p_{\text{min}}, 2)$. }\label{fig:bifp}
\end{center}
\end{figure}

\begin{figure}[h!t]
\begin{center}
\includegraphics[width=\textwidth]{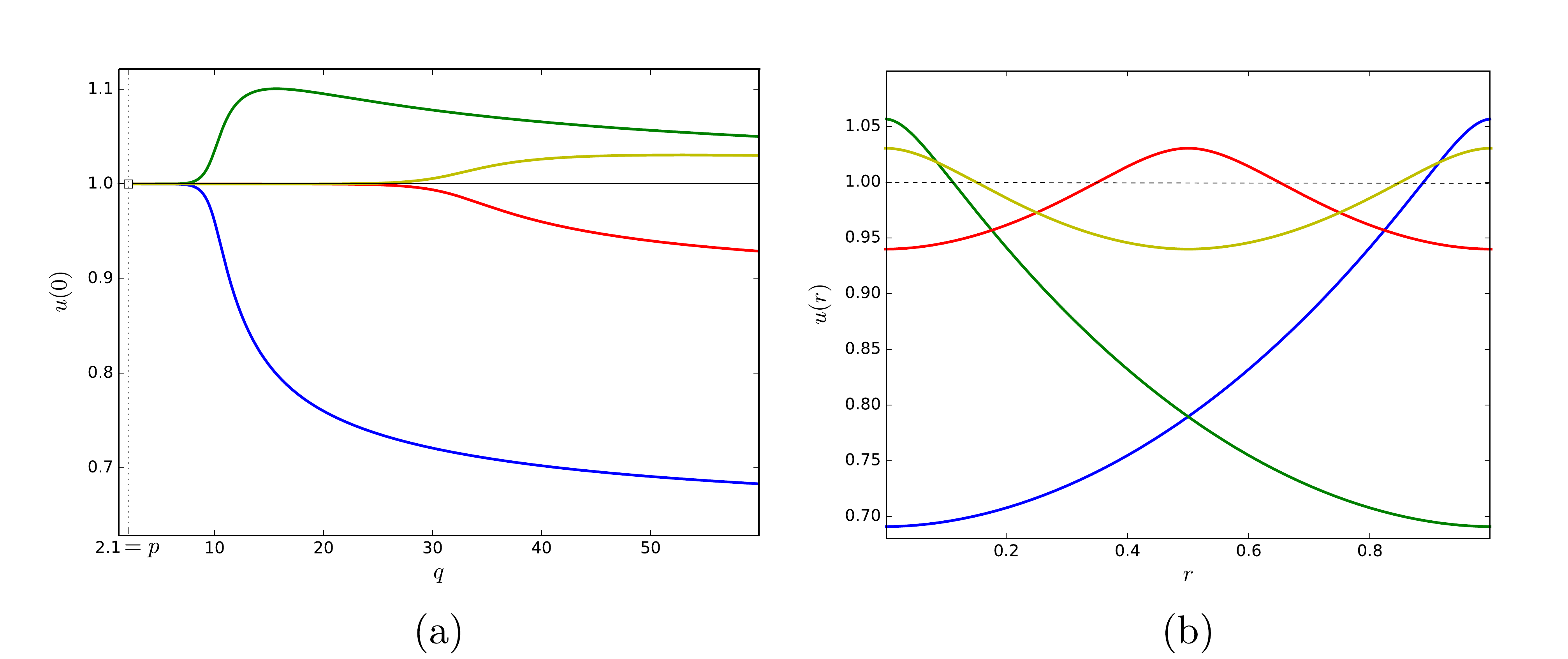}
\caption{Partial bifurcation diagram for problem \eqref{eq:modelloN1} with $p=2.1$: the first branches of solutions bifurcating at $q=p$. The color of each solution in the right plot corresponds to the color of the branch it belongs to in the left plot. (a) Bifurcation diagram: $u(0)$ as function of $q$. (b) Solutions belonging to the first branches for $q\simeq 50$.}
\label{fig:ps2}
\end{center}
\end{figure}

We wish to mention that the solutions presented in Figure~\ref{fig:pi2} could not be detected by bifurcating from the constant. Instead, we adopted the following technique. We started from a solution of the problem with $p=2$ and $q$ equal to a certain value $\bar q$. We considered bifurcation for this solution in the parameter $p$ (with $q=\bar q$ fixed). This provides a continuum of solutions up to a certain minimum value $p_{\text{min}}$, as shown for example in Figure~\ref{fig:bifp}. The two solutions obtained in this way for a certain $\bar p$ satisfying $p_{\text{min}}\leq \bar p<2$ can be used as a starting point to obtain the graph in Figure~\ref{fig:pi2} (with $p=\bar p$ fixed and $q$ variable). According to this discussion, it seems to be an interesting question whether the multiplicity scheme of Theorem \ref{th:main2} could be obtained when varying $p$ instead of the diameter of the domain, that is if, given a domain and given $k\geq1$, it is possible to obtain $2k$ solutions for any $p\in(p_{\textrm{min}}(k),2)$.

Finally, when $p>2$, there seems to persist a phenomenon of bifurcation from the constant solution.
We conjecture that in this case infinite branches bifurcate from the same point $q=p$, giving rise to a very degenerate situation. Notice that this would be coherent with the result of Corollary \ref{cor:modello}-(i). In Figure~\ref{fig:ps2} we present the first bifurcation branches for $p=2.1$. Some numerical difficulties occur also in this case, probably due to the fact that an infinite number of curves meet at $q=p$ and that these curves start with an almost flat shape. In order to detect the blue and green branches, we took advantage of the monotonicity of the solutions belonging to them; to find the other two branches we adopted the method described above for the case $p<2$.
Also in this case, the existence of solutions satisfying $u(0)>1$ is an open problem.

\section*{Acknowledgments}
\noindent A. Boscaggin acknowledges the support of the projects MIS F.4508.14 (FNRS) \& ARC AUWB-2012-12/17-ULB1- IAPAS for his visits at Universit\'e Libre de Bruxelles, where parts of this work have been achieved.
A. Boscaggin and F. Colasuonno were partially supported by the INdAM - GNAMPA Projects 2016 ``Problemi differenziali non lineari: esistenza, molteplicità e
proprietà qualitative delle soluzioni'' and ``Fenomeni non-locali: teoria, metodi e applicazioni", respectively. Furthermore, A. Boscaggin and B. Noris were also supported by the project ERC Advanced Grant 2013 n. 339958: ``Complex Patterns for Strongly Interacting Dynamical Systems --
COMPAT''. B. Noris wishes to thank A. Salda\~na and G. Petretto for their help with AUTO07p and  Python respectively.


\def\cprime{$'$}

\end{document}